\documentclass[12pt]{article}
\usepackage{mathrsfs}
\usepackage{latexsym}
\usepackage{amsmath, amsthm, amssymb, amstext}
\usepackage{authblk,color}
\usepackage{cite}

\usepackage{enumitem}

\usepackage{hyperref}
\input amssym.def
\input amssym

\renewcommand \it[1] {\textit{#1}}

\newcommand \al {\alpha}
\newcommand \g {\gamma}

\renewcommand \l {\lambda}

\renewcommand \c {\mathbb C}

\newcommand \f {\mathbb F}
\newcommand \mf {\mathcal F} 

\newcommand \si {\sigma}

\newcommand \ol[1]{\overline{#1}}
\newcommand \wh[1]{\widehat{#1}}

\newcommand \lrg[1]{\langle #1 \rangle}
\newcommand \abs[1]{\lvert #1 \rvert}
\newcommand \bigabs[1]{\big \lvert #1 \big \rvert}
\newcommand \lrabs[1]{\left \lvert #1 \right \rvert}
\newcommand \nid{\, : \, }

\newcommand \four{Fourier transform}
\newcommand \gdual{$G$-dual set}
\newcommand \sgd{normalized \gdual}
\newcommand \fag{finite abelian group}
\newcommand \fg{finite group}
\newcommand \pn{perfect nonlinear}
\newcommand \gpn{$G$-perfect nonlinear}
\newcommand \pnf{perfect nonlinear function}
\newcommand \gpnf{$G$-perfect nonlinear function}
\newcommand \gbent{$G$-bent}
\newcommand \gbf{$G$-bent function}
\newcommand \ic{irreducible character}

\newcommand \gi{G \setminus \{ 1_G \}}

\newcommand \dset{difference set}
\newcommand \dsx{difference set of $X$}
\newcommand \rdf{related difference family}
\newcommand \ghrdf{$(G, H)$-\rdf}
\DeclareMathOperator{\irr}{Irr}

\newtheorem{thm}{Theorem}[section]
\newtheorem{defn}[thm]{Definition}
\newtheorem{lemma}[thm]{Lemma}

\newtheorem{cor}[thm]{Corollary}

\newtheorem{ex}[thm]{Example}
\newtheorem{nota}[thm]{Notation}
\newtheorem{re}[thm]{Remark}

\numberwithin{equation}{section}
\numberwithin{table}{section}

\title{Nonlinear functions and difference sets on group actions}

\begin{document}

\author[1]{Yun Fan} \author[2]{Bangteng Xu}
\affil[1]{{\small School of Mathematics and Statistics},
 {\small Central China Normal University, Wuhan 430079, China}}
\affil[2]{{\small Department of Mathematics and Statistics},
{\small Eastern Kentucky University, Richmond, KY 40475, USA}}

\date{}

\insert\footins{\footnotesize
\it{Email addresses}:  yfan@mail.ccnu.edu.cn (Y. Fan),
bangteng.xu@eku.edu (B. Xu, corresponding author)}

\maketitle

\begin{abstract}

Let $G$, $H$ be finite groups and let $X$ be a finite $G$-set. 
$G$-perfect nonlinear functions from $X$ to $H$ 
have been studied in several papers. They have more interesting properties 
than perfect nonlinear functions from $G$ itself to $H$. 
By introducing the concept of a $(G, H)$-related difference family of 
$X$, we obtain a characterization of $G$-perfect nonlinear 
functions on $X$. When $G$ is abelian, we characterize a $G$-difference set of $X$
by the Fourier transform on a normalized $G$-dual set $\widehat X$. 
We will also investigate  
the existence and constructions of $G$-perfect nonlinear functions and 
$G$-bent functions. Several known results in \cite{ad, dp, ph, xu1}
are direct consequences of our results.

\vskip 8pt

\it{Key Words\,:} $G$-perfect nonlinear functions; $G$-difference sets; $(G, H)$-related
difference families; normalized $G$-dual sets; Fourier transforms 
\vskip 8pt

\it{Mathematics Subject Classification:} 05B10, 05E18, 65T50, 94E18
\end{abstract}

%%%%%%%%%%%%%%%%%%%%%%%%%

\section{\large Introduction 
\label{sect-int} }

%%%%%%%%%%%%%%%%%%%%%%%%%

Perfect nonlinear functions are actively studied by many researchers
(cf. \cite{cd, dp, fx, lsy, ph, p06, pp, pott, s02, xu1}). 
They can be used to construct DES-like cryptosystems that are resistant to differential
attacks. \gpnf s have been studied in several papers. They are natural generalizations
of the classical \pnf s with the XOR operations replaced by group actions, and have
more interesting properties. They can also be 
used to construct generalized DES-like cryptosystems (where the XOR operation is 
replaced by a group action) that are resistant to modified differential attacks.
For a summary of the background on cryptosystems and group action modifications,
the reader is referred to \cite{dp}. Pott \cite{pott} mentioned that ``It seems that in
most applications (in particular in cryptography) people use nonlinear functions on
finite fields. However, there is no technical reason why you should restrict yourselves 
to this case".

In their original research on \gbent\ and \gpnf s on a $G$-set $X$,
Poinsot et al. \cite{ph, p09}  
characterize these functions by \four s of functions on $G$.  Davis et al. \cite{dp}
characterize \gpnf s on a $G$-set $X$ by $G$-\dset s of $X$, and construct \gpnf s
by constructing $G$-\dset s, without using \four s at all. For a \fag\ $G$ and a 
$G$-set $X$, a \gdual\ $\wh X$ of $X$ is introduced in \cite{fx}, which plays a role
similar to the dual group $\wh G$. Then the \four s of functions 
 on $X$ are defined as functions on $\wh X$, and used to characterize  
\gbent\ and \gpnf s in \cite{fx}. 

In this paper we will first give a characterization of \gpnf s on a $G$-set $X$
 by $(G, H)$-related difference families of $X$ (see Theorem \ref{pnf-rdf}).
As applications of this characterization,
we obtain several known results in \cite{ad, dp, xu1} as immediate consequences.
When $G$ is a \fag, we introduce the concept of a \sgd\ (see Definition \ref{defn-sgd}),
and use it to characterize a $G$-\dset\ of $X$ (see Theorem \ref{thm-set-char}).
Using the method developed in this paper, we are able to give a very short and easy proof
of an interesting result of Poinsot et al. \cite{ph} (see Corollary \ref{cor-reg}).
Furthermore, by applying the \four s on \sgd s and 
the characterizations of \gpnf s established in \cite{fx}, we investigate the existence
and constructions of \gpnf s and \gbf s on $X$(see Theorems \ref{thm-order2}, 
\ref{thm-pnf}, \ref{thm-pnf-gen}, and \ref{thm-klein4}). The methods developed in 
\cite{fx} and in this paper provide new insights into the study of \gbent\ and \gpnf s.  

The rest of the paper is organized as follows. 
In Section \ref{related df} we present the characterization of $G$-perfect nonlinear functions
by $(G, H)$-related difference families, and discuss its applications. 
In Section \ref{sect-pre}, after the brief recall of some results 
 in \cite{fx} that are needed later, we
introduce the notion of a \sgd, and study its basic properties. 
The characterization of $G$-\dset s by \sgd s is given in Section \ref{sect-set}.
Sections \ref{sect-klein} and \ref{sect-klein-b} are devoted to 
the study of the existence and constructions of \gpnf s and \gbf s 
for the Klein four group $G$, respectively.

%%%%%%%%%%%%%%%%%%%%%%%%

\section{\large Perfect nonlinear functions and related deference families
\label{related df}}

%%%%%%%%%%%%%%%%%%%%%%%%

In this section we introduce the concept of a 
$(G,H)$-related difference family (see Definition \ref{rdf}), and use it 
to characterize $G$-perfect nonlinear functions (see Theorem \ref{pnf-rdf}).
Applications of Theorem \ref{pnf-rdf} will also be discussed, and several known 
results in \cite{ad, dp, xu1} are obtained as direct consequences.
Throughout the paper, the following notation will be used. 

\begin{nota}
\label{assume}\rm
Let $G$ be a finite group of order $m$, and $X$ a finite $G$-set of cardinality $v$. 
That is, there is a map $G\times X\to X$, $(\alpha,x)\mapsto\alpha x$,
such that for any $x\in X$,  $(\alpha\beta) x = \alpha(\beta x)$ 
for all $\alpha,\beta\in G$, and $1_G x = x$,
where $1_G$ is the identity element of $G$. Let $H$ be a finite group. 
\end{nota}

For the fundamentals of group actions, the reader is referred to \cite{ab}.

Let $f:X\to H$ be a function.
For any $\alpha\in G$ the {\em derivative of $f$ in direction $\alpha$},
denoted by $f'_\alpha$, is defined by
$$
f'_\alpha:~ X\to H, \quad  x\mapsto f(\alpha x)f(x)^{-1}.
$$ 
The cardinality of a set $S$ is denoted by $\abs{S}$. As usual (cf. \cite{ph,dp}), 
the function $f$ is said to be {\em $G$-perfect nonlinear} if $|H|$ divides $|X|$ and
\begin{equation} 
\label{pnf}
\big| {f'_\alpha}^{-1}(\sigma)\big|=\frac{|X|}{|H|},\qquad
\forall~\alpha\in \gi ~~\forall~\sigma\in H,
\end{equation}
where ${f'_\alpha}^{-1}(\sigma) :=\{x\in X \nid f'_\alpha(x)=\sigma\}$
is the inverse image of $\sigma$ in $X$.

As generalizations of the usual difference sets and the relative difference sets of groups,
Davis et al. \cite{dp} introduced the following definition. 
Note that $G \times H$ acts on $X \times H$ by 
$(\alpha, \sigma)(x, h) = (\alpha x, \sigma h)$, where 
$(\alpha, \sigma) \in G \times H$ and $(x, h) \in X \times H$.

\begin{defn} [{Cf. \cite[Definition 3.2]{dp}}] \rm
\begin{enumerate}
\item[$($i$)$] A subset $D$ of cardinality $k$ of $X$ is called a 
{\em $G$--$(v,k,\ell)$ difference set} of $X$ 
if for any $\alpha\in \gi$ ,
there are exactly $\ell$ elements $(x_1, x_2)$ of $D\times D$ such that 
$\alpha x_1=x_2$.

\item[$($ii$)$] A subset $R$ of cardinality $k$ of $X\times H$ is called a 
{\em $G\times H$--$(v,|H|,k,\ell)$-relative difference set of 
$X\times H$ relative to $\{1_G\}\times H$} if (i)
for every $(\alpha,\sigma)\in (G\times H) \setminus (\{1_G\}\times H)$
there are exactly $\ell$ elements $\big((x_1, h_1), (x_2, h_2)\big)\in R\times R$ 
such that $(\alpha, \sigma)(x_1, h_1) = (x_2, h_2)$ and (ii) if $(x,h),(x,h') \in R$,
 then $h = h'$. 
Such a $G\times H$--$(v, |H|, k, \ell)$-relative difference set is said to be
{\em semiregular} if $v = k$.
\end{enumerate}
\end{defn}

Davis et al. \cite{dp} characterized
$G$-perfect nonlinear functions  from $X$ to $H$ by $G$-difference sets 
(when $|H|=2$, see Corollary \ref{cor-dp-2-pnf} below)
and by semiregular relative difference sets (see Corollary \ref{c-dp4} below).

Note that for any $\al \in G$ and subsets $C, D$ of the $G$-set $X$, 
\begin{equation}
\label{alpha x=y}
\big|\{(x,y)\in C\times D \nid \alpha x=y\}\big|=|\alpha C \cap D|,
\end{equation}
where $\alpha C := \{\alpha x \nid x\in C\}$.
Thus, a subset $D$ of cardinality $k$ of $X$ is a $G$--$(v,k,\ell)$ difference set 
if and only if
$|\alpha D\cap D|=\ell$ for all $\alpha\in \gi$.
Since any $\alpha\in G$ induces a permutation on~$X$, we have
\begin{equation}\label{alpha^-1}
|\alpha C\cap D|=|\alpha^{-1}(\alpha C\cap D)|=|C\cap\alpha^{-1}D|.
\end{equation}
In particular, $\abs{\alpha D \cap D} = 
\abs{\alpha^{-1} D \cap D}$, for any $\al \in G$. Let $D':=X\setminus D$.
Then $(\alpha D\cap D) \cup (\alpha D\cap D')=\alpha D \cap (D\cup D')
=\alpha D\cap X =\alpha D$.
So $D$ and $D'$ disjoint implies that
\begin{equation}
\label{D}
|\alpha D\cap D|+|\alpha D\cap D'|=|D|.
\end{equation}
Now it follows from Eqn \eqref{D} and Eqn \eqref{alpha^-1} that
$$
|\alpha D\cap D'| = |D|-|\alpha D\cap D|
 = |D|-|\alpha^{-1}D\cap D|=|\alpha^{-1} D\cap D'|=|D\cap\alpha D'|.
$$
That is,
\begin{equation}\label{DD'}
|\alpha D\cap D'|=|D\cap\alpha D'|,\qquad\forall~\alpha\in G,
\ \hbox{where } D' = X \setminus D.
\end{equation}

\begin{lemma}\label{ds-DD'}
Let $D$ be a subset of cardinality $k$ of the $G$-set $X$, 
and $D':=X\setminus D$ with cardinality $k' := v - k$. Then
the following are equivalent.

\begin{enumerate}
\item[$($i$)$] 
 There is a nonnegative integer $w$ such that 
$|\alpha D\cap D'|=w$ $($or equivalently, $|D\cap\alpha D'|=w)$
 for all $\alpha\in G\setminus\{1_G\}$.

\item[$($ii$)$] $D$ is a $G$--$(v,k,k-w)$ difference set.

\item[$($iii$)$] $D'$ is a $G$--$(v,k',k'-w)$ difference set.
\end{enumerate}
\end{lemma}

\begin{proof}
By Eqn \eqref{D}, for all $\alpha\in G\setminus\{1_G\}$,
$|\alpha D\cap D'|=w$  if and only if $|\alpha D\cap D|=k-w$. So 
(i) and (ii) are equivalent. Similarly, (i) and (iii) are also equivalent.
\end{proof}

Xu \cite{xu1} introduced the concept of a 
 $(G,H)$-related difference family of a group $G$ 
(which is different from the difference families defined in \cite{bjl})
and used it to characterize 
$G$-perfect nonlinear functions on groups.
We generalize this concept to $G$-sets as follows.

\begin{defn}\label{rdf}\rm
Let $S_h\subseteq X$ for $h\in H$ be disjoint subsets
indexed by $H$ such that $\bigcup_{h\in H}S_h=X$. 
We say that $\{S_h \nid h\in H\}$ is a {\em $(G,H)$-related difference family} of $X$ if 
for any $\alpha\in G\setminus\{1_G\}$ and
$\sigma\in H\setminus\{1_H\}$ there are exactly $\frac{|X|}{|H|}$ elements 
$(x,y)\in \bigcup_{h\in H}(S_h\times S_{\sigma h})$
such that $\alpha x=y$.
\end{defn}

Note that if there exits a $(G,H)$-related difference family of $X$, then $\abs{H}$ divides
$\abs{X}$. Also in the above definition, we do not assume that $S_h$ is nonempty for 
every $h \in H$.

\begin{lemma}\label{rdf'}
Let $S_h\subseteq X$ for $h\in H$ be disjoint subsets
such that $\bigcup_{h\in H}S_h=X$.
Then the following are equivalent.
\begin{enumerate}
\item[$($i$)$] 
 $\{S_h \nid h\in H\}$ is a $(G,H)$-related difference family.

\item[$($ii$)$] 
$\sum_{h\in H}\big|\alpha S_h\cap S_{\sigma h}\big|=\frac{|X|}{|H|}$~
for all $\alpha\in G\setminus\{1_G\}$ and $\sigma\in H\setminus\{1_H\}$. 

\item[$($iii$)$]  
$\sum_{h\in H}\big|\alpha S_h\cap S_{\sigma h}\big|=\frac{|X|}{|H|}$~
for all $\alpha\in G\setminus\{1_G\}$ and all $\sigma\in H$. 
\end{enumerate}
\end{lemma}

\begin{proof} 
It follows from $S_h\subseteq X$ disjoint for all $h\in H$ that 
for any $\al \in G$, 
$\alpha S_h \cap S_{\sigma h}$ are also disjoint for all $h, \si \in H$,
and $S_h\times S_{\sigma h}$ are disjoint for all $h, \si \in H$, too.
So by Eqn \eqref{alpha x=y}, the number of elements 
$(x,y)\in \bigcup_{h\in H}(S_h\times S_{\sigma h})$
such that $\alpha x=y$ is equal to 
$\sum_{h\in H}\big|\alpha S_h\cap S_{\sigma h}\big|$.
Thus, (i) and (ii) are equivalent. Furthermore,  since $\bigcup_{h\in H}S_h=X$,
$$
\bigcup_{\sigma\in H}(\alpha S_h\cap S_{\sigma h})=
\alpha S_h \bigcap
\Big(\bigcup_{\sigma\in H} S_{\sigma h}\Big)= \al S_h \cap X= \al S_h.
$$
Hence, $S_{\si h}$ disjoint for all $\si \in H$ imply that 
$\sum_{\sigma\in H} \big|\alpha S_h\cap S_{\sigma h}\big| = \abs{S_h}$.
Therefore,
$$
\sum_{\sigma\in H}\sum_{h\in H}\big|\alpha S_h\cap S_{\sigma h}\big|=|X|.
$$
So (ii) and (iii) are equivalent.
\end{proof}

\begin{thm}\label{pnf-rdf}
Let $f: X \to H$ be a function, and $S_h := f^{-1}(h)$ 
for all $h \in H$. Then $f$ is a $G$-perfect nonlinear function
if and only if
$\{ S_h \nid h \in H \}$ is a $(G, H)$-related difference family.
\end{thm}

\begin{proof} Clearly the subsets $S_h$ of $X$  
are disjoint for all $h\in H$, and  $\bigcup_{h\in H}S_h=X$.  
Let $\alpha\in G\setminus\{1_G\}$ and $\sigma, h \in H$. Then for any
 $x \in X$,  
\begin{eqnarray*}
x \in {f'_\alpha}^{-1}(\sigma) \cap S_h
& \iff & f(\alpha x)f(x)^{-1}=\sigma \hbox{ and } f(x) = h \\
&\iff& f(\alpha x)=\sigma f(x) = \sigma h 
\ \iff \ \alpha x\in\alpha S_h\cap S_{\sigma h}.
\end{eqnarray*}
Thus, $\bigabs{{f'_\alpha}^{-1}(\sigma) \cap S_h} = 
\bigabs{ \alpha S_h\cap S_{\sigma h}}$, for any
$\alpha\in G\setminus\{1_G\}$ and $\sigma, h \in H$. Hence,
$$
\big|{f'_\alpha}^{-1}(\sigma)\big|=
\sum_{h\in H} \bigabs{{f'_\alpha}^{-1}(\sigma) \cap S_h}
= \sum_{h\in H}\big|\alpha S_h\cap S_{\sigma h}\big|,
\qquad\forall~ \al \in \gi, \ \ \forall~  \sigma\in H.
$$
By the definition of a \gpnf\ (cf. Eqn \eqref{pnf}) and Lemma \ref{rdf'}, the theorem holds. 
\end{proof}

As direct consequences, we have the following corollaries.

\begin{cor}
[Cf. \cite{xu1}, Theorem 1.3]
Let $f: G \to H$ be a function, and $S_h := f^{-1}(h)$, for any $h \in H$. Then the
following are equivalent.
\begin{enumerate}
\item[$($i$)$] $f$ is \pn.
\item[$($ii$)$] $\{ S_h \nid h \in H \}$ is a \ghrdf.
\end{enumerate}
\end{cor}

\begin{cor}[{Cf. \cite[Theorem 14]{ad}, \cite[Theorem 3.4]{dp}}]
\label{c-dp4}
Let $f: X \to H$ be a function, and let 
$R_f=\big\{\big(x,f(x)\big)\,\big|\, x\in X\big\}\subseteq X\times H$.
Then $f$ is $G$-perfect nonlinear if and only if $R_f$ is a 
$G\times H$--$(|X|,|H|,|X|,\frac{|X|}{|H|})$ semiregular relative difference set of 
$X\times H$ relative to $\{1_G\}\times H$.
\end{cor}

\begin{proof} Let $S_h := f^{-1}(h)\subseteq X$ for $h\in H$, and 
let $\al \in \gi$ and $\si \in H$. Then for any $(y, f(y)) \in R_f$, 
\begin{eqnarray*}
(y, f(y)) \in (\al, \si)  R_f \cap R_f & \iff &
   (y, f(y)) = (\al x, \si f(x)) \ \hbox{ for some (unique) } x \in X \\
& \iff & y \in \al S_h \cap S_{\si h} \ \hbox{ for some (unique) } h \in H. 
\end{eqnarray*}
Thus, 
$$
\big|(\alpha, \sigma) R_f \cap R_f \big|=
\sum_{h\in H}\big|\alpha S_h\cap S_{\sigma h}\big|,
\qquad\forall \alpha\in G\setminus\{1_G\}~~\forall~\sigma\in H.
$$
So the corollary holds by Eqn \eqref{alpha x=y}, Lemma \ref{rdf'}, and
Theorem \ref{pnf-rdf}.
\end{proof}

The difference families of groups in \cite{bjl} can be generalized to
$G$-difference families of $G$-sets as follows. 
Let $\mf := \{ S_1, \dotsc, S_p \}$ be
a family of nonempty subsets of $X$, and $K := \{ \abs{S_i} \nid 1 \le i \le p \}$.
If for any $\al \in \gi$, there are exactly $\ell$ elements $(x, y)$ in
$
\bigcup_{i=1}^p (S_i \times S_i)
$
such that $\al x=y$, then $\mf$ is called a {\em $G$--$(v, K, \ell)$ difference family}
of $X$. Furthermore, if $\mf$ forms a partition of $X$, then $\mf$ is called a
\it{partitioned $G$--$(v, K, \ell)$ difference family} of $X$. 
As an immediate consequence of Theorem \ref{pnf-rdf}, we have the following

\begin{cor}
Let $f: X \to H$ be a \gpnf. Let $S_h := f^{-1}(h)$ for any $h \in H$, 
and $K := \{ \abs{S_h} \nid h \in f(X) \}$, where $f(X)$ denotes the image of $f$. 
Then $\{ S_h \nid h \in f(X) \}$
is a partitioned $G$--$(\abs{X}, K, \abs{X}/\abs{H})$ difference family in $X$.
\end{cor}

The relation between $G$-perfect nonlinear functions and $G$-difference sets 
established in \cite{dp}
can also be easily obtained as a direct consequence of Theorem \ref{pnf-rdf} 
and Lemma \ref{ds-DD'}.

\begin{cor}
[Cf. \cite{dp}, Theorem 3.3]
\label{cor-dp-2-pnf}
Let $f: X \to {\mathbb F}_2$ be a function, where ${\Bbb F}_2$ denotes the field
of $2$ elements (but here it is regarded as an additive group of order $2$).
Let $S_0= f^{-1}(0)$, $S_1= f^{-1}(1)$,
$k_0=|S_0|$, and $k_1=|S_1|$ (hence $k_0+k_1=v$). 
Then $f$ is a $G$-perfect nonlinear function
if and only if 
$4|v$ and $S_1$ is a $G$--$(v,k_1,k_1-\frac{v}{4})$ difference set
(or equivalently,  $S_0$ is a $G$--$(v,k_0,k_0-\frac{v}{4})$ difference~set).
\end{cor}

\begin{proof} By Theorem \ref{pnf-rdf}, 
$f$ is $G$-perfect nonlinear if and only if
$$
|\alpha S_0\cap S_1|+|\alpha S_1 \cap S_0|=\frac{|X|}{|{\Bbb F}_2|}=\frac{v}{2},
\qquad\forall~\alpha\in \gi.
$$
By Eqn \eqref{DD'}, the above equality is equivalent to that
$$
|\alpha S_0\cap S_1|=|S_0\cap\alpha S_1|=\frac{v}{4},
\qquad\forall~\alpha\in \gi.
$$
Thus the corollary follows from  Lemma \ref{ds-DD'} immediately.
\end{proof}

Moreover, by applying Theorem \ref{pnf-rdf} and Lemma \ref{ds-DD'}, 
in the next theorem 
we can determine all $G$-perfect nonlinear functions on  a $G$-set 
whose images contain exactly two elements.
Note that examples of \gpnf s that satisfy the properties described in the 
next theorem can be found in \cite{dp} and  \cite[Example 6.5]{fx}.

\begin{thm} 
\label{thm-2-3}
Assume that $f:X\to H$ is a function 
whose image $f(X)$ consists of exactly two elements $h_1,h_2\in H$.
Let $S_h := f^{-1}(h)$ for $h\in H$, and $k_i := |S_{h_i}|$ for $i=1,2$. 
Then $f$ is a $G$-perfect nonlinear function
if and only if one of the following holds.
\begin{enumerate}
\item[$($i$)$] 
 $|H|=2$, $4|v$ and $S_{h_1}$ is a $G$--$(v,k_1,k_1-v/4)$ difference set
(or equivalently,  $S_{h_2}$ is a $G$--$(v,k_2,k_2-v/4)$ difference set).

\item[$($ii$)$] $|H|=3$, $3|v$ and $S_{h_1}$ is a $G$--$(v,k_1,k_1-v/3)$ difference set
(or equivalently,  $S_{h_2}$ is a $G$--$(v,k_2,k_2-v/3)$ difference set).
\end{enumerate}
\end{thm}

\begin{proof} Note that both $k_1$ and $k_2$ are positive, but
$|S_h|=0$ (i.e. $S_h=\emptyset$) for $h\ne h_1,h_2$.

Suppose that $f$ is a $G$-perfect nonlinear function. 
By Theorem \ref{pnf-rdf}, for any $\alpha\in G\setminus\{1_G\}$, we have
\begin{equation*}
\sum_{h\in H}\big|\alpha S_h\cap S_{\sigma h}\big|=\frac{|X|}{|H|},
\qquad\forall~\sigma\in H\setminus\{1_H\}.
\end{equation*}
If $h\ne h_1,h_2$, then $\alpha S_h=\emptyset$, and hence
$\big|\alpha S_h\cap S_{\sigma h}\big|=0$. So  
the above equality can be reduced to
\begin{equation}\label{two terms}
 \big|\alpha S_{h_1}\cap S_{\sigma h_1}\big|+
 \big|\alpha S_{h_2}\cap  S_{\sigma h_2}\big|=\frac{|X|}{|H|},
 \qquad\forall~\sigma\in H\setminus\{1_H\}.
\end{equation}
Since $\sigma\ne 1_H$, $\sigma h_i\ne h_i$ for $i=1,2$. 
If $\sigma h_1\ne h_2$ and $\sigma h_2\ne h_1$, then
$S_{\sigma h_1}=\emptyset=S_{\sigma h_2}$, and
Eqn~\eqref{two terms} cannot hold. 

Thus, for any $\sigma\in H\setminus\{1_H\}$, we have 
either $\sigma h_1=h_2$ (i.e. $\sigma=h_2h_1^{-1}$) 
or $\sigma h_2=h_1$ (i.e. $\sigma=h_1h_2^{-1}$).
Note that it is possible $h_2 h_1^{-1} = h_1h_2^{-1}$.

If  $h_2 h_1^{-1} = h_1h_2^{-1}$, then for any
$\sigma\in H\setminus\{1_H\}$, we must have 
$\si = h_2h_1^{-1}=h_1h_2^{-1}$.  So $H=\{1,\sigma\}$, and (i) holds
by Corollary \ref{cor-dp-2-pnf}.

Now assume that $h_2h_1^{-1}\ne h_1h_2^{-1}$. 
Since $h_1\ne h_2$, none of $h_2h_1^{-1}$ and  $h_1h_2^{-1}$ is the identity 
element. Because we have already shown that any non-identity element of $H$ must be either 
$h_2h_1^{-1}$ or $h_1h_2^{-1}$, we see that $|H|=3$. Thus, if $\si = 
h_2h_1^{-1}$, then $\si h_1 = h_2$, and $\si h_2 \notin \{h_1, h_2 \}$.
If $\si = h_1 h_2^{-1}$, then $\si h_2 = h_1$, and
$\si h_1 \notin \{h_1, h_2 \}$. Therefore,
 Eqn \eqref{two terms} becomes
\begin{equation}\label{order 3}
 \big|\alpha S_{h_1}\cap S_{h_2}\big|=\big|S_{h_1}\cap\alpha  S_{h_2}\big|
 =\frac{|X|}{|H|}=\frac{v}{3}, \quad \forall~\sigma\in H\setminus\{1_H\}.
\end{equation}
By Lemma \ref{ds-DD'}, 
$S_{h_1}$ is a $G$--$(v,k_1,k_1-v/3)$ difference set,
which is equivalent to that 
$S_{h_2}$ is a $G$--$(v,k_2,k_2-v/3)$ difference set.

\smallskip 
Conversely, if (i) holds, then  
$f$ is a $G$-perfect nonlinear function by Corollary \ref{cor-dp-2-pnf}.
If (ii) holds, then by Lemma \ref{ds-DD'}, 
Eqn \eqref{order 3} holds. Thus,   
$f$ is a $G$-perfect nonlinear function by Theorem \ref{pnf-rdf}.
\end{proof}

%%%%%%%%%%%%%%%%%%%%%%%%%
\section{\large Normalized $G$-dual sets 
\label{sect-pre} }
%%%%%%%%%%%%%%%%%%%%%%%%%

With Notation \ref{assume},  in the rest of the paper 
we always assume that $G$ and $H$ are abelian. Then we
study \gpnf s, $G$-bent functions, and $G$-difference sets, etc.
by using Fourier transforms on $G$-sets.  

In this section we first briefly state some known results that will be 
needed later. Then we introduce the concept of a normalized $G$-dual set 
and discuss its basic properties.
For bent and \pnf s on \fg s, the reader is referred to
\cite{cd, d, lsy, p06, pp, pott, s02, xu1, xu2, xu3}.

Let $\wh G$ be the dual group of $G$. That is, $\wh G$ is the group of \ic s
of $G$ over the complex field $\c$.
The \it{principal \ic} of $G$ is the character $\psi_0: G \to \c, \al
\mapsto 1$. The kernel of $\psi \in \wh G$ is 
$\ker \psi := \{\al\in G \nid \psi(\al) = 1 \}$. Note that $\ker \psi$ is a subgroup of $G$. 
For any $z \in \c$, let $\ol{z}$ be the complex conjugate of $z$, and $\abs{z}$
the absolute value (or complex modulus) of $z$. Then for any $\psi \in \wh G$ and 
$\al \in G$, $\abs{\psi(\al)} = 1$. Furthermore, for any $\psi, \varphi \in \wh G$,
$$
\sum_{\al \in G} \psi(\al) \ol{\varphi(\al)} =
\begin{cases}
\abs{G}, & \hbox{if } \psi = \varphi; \\
0, & \hbox{if } \psi \ne \varphi,
\end{cases}
$$
(the First Orthogonality Relation of \ic s).
For the reference of the character theory of \fg s, the reader is referred to \cite{h, s}.

Let $\c^X$ be the set of all functions from $X$ to $\c$. Then $\c^X$
is a $G$-space with the $G$-action defined by
$$
(\al f)(x) := f(\al^{-1} x), \quad \hbox{for any } f \in \c^X, \al \in G, x \in X. 
$$
For any $f \in \c^X$, if there is a $\psi \in \wh G$ such that
$$
f(\al^{-1} x) = \psi(\al) f(x), \quad \hbox{for any } \al \in G, x \in X,
$$ 
then $f$ is said to be \it{$\psi$-linear} or \it{$G$-linear} (cf. 
\cite[Definition 2.1]{fx}).  
 Also $\c^X$ is a unitary space with inner product
$$
\lrg{f, g} := \sum_{x \in X} f(x) \ol{g(x)}, \quad \hbox{for any } f, g \in \c^X. 
$$
For a positive real number $\mu$, 
a basis $\{u_1, \dotsc, u_v \}$ of $\c^X$ is called a \it{$\mu$-normal orthogonal 
basis} if $\lrg{u_i, u_j} = \delta_{ij} \mu$ for all $1 \le i, j \le v$, where 
$\delta_{ij}$ is the Kronecker delta.  The complex conjugate of  
$f \in \c^X$ is $\ol f \in \c^X$ defined by $\ol{f}(x) := \ol{f(x)}$, for any
$x \in X$.

\begin{defn} \rm
(Cf. \cite[Definition 2.2]{fx})
A \it{\gdual} $\wh X$ of a $G$-set $X$ is a
$\abs{X}$-normal orthogonal basis of $\c^X$ such that any $\l \in \wh X$
is $G$-linear and $\wh X$ is closed under complex conjugate (i.e. 
$\ol{\l} \in \wh X$ for any $\l \in \wh X$).
\end{defn}

For any $G$-set $X$, there exists a \gdual\ $\wh X$ by \cite[Theorem 2.3]{fx}.  
Let $\wh X$ be a \gdual\ of~$X$. For any $\psi \in \wh G$, let $\wh X_\psi := 
\{ \l \in \wh X \nid \l $ is $\psi$-linear\}. Then $\wh X_\psi$ are disjoint for
all $\psi \in \wh G$ and $\wh X = \bigcup_{\psi \in \wh G} \wh X_\psi$.

\begin{defn} \rm
(Cf. \cite[Definition 3.1]{fx})
For any $f \in \c^X$, the {\em Fourier transform} of $f$ on $\wh X$ is a function
 $\wh f \in \c^{\wh X}$ defined by
$$
\wh f(\l)=\sum_{x \in X} f(x) \l(x), \quad \hbox{for any }\l \in \wh X.
$$
\end{defn}

Let $T := \{ z \in \c \nid \abs{z} = 1 \}$ be the unit circle in $\c$.
A function $f:X\to T$ is called a {\em \gbf} (cf. \cite[Definition 4.1]{fx}) if
\begin{equation}
\label{def-bent}
\sum_{\l \in \wh X_{\psi}} \bigabs{ \wh f(\l) }^2
=\frac{|X|^2}{|G|}, \quad \hbox{for all } \psi \in \wh G.
\end{equation}
Since $\sum_{\psi\in\wh G}\sum_{\l\in\wh X_\psi}|\wh f(\l)|^2
 =|\wh f|^2=|X|^2$ (see \cite[Corollary 3.5]{fx}), 
it follows that 
\begin{equation}
\label{def-bent'}
f \hbox{ is $G$-bent } \quad \iff \quad
\sum_{\l \in \wh X_{\psi}} \bigabs{ \wh f(\l) }^2
=\frac{|X|^2}{|G|}, \quad \hbox{for all } \psi 
\in \wh G\setminus\{\psi_0\}.
\end{equation}

Let $f: X \to T$ be a function, and $\al \in G$. Then the \it{derivative of $f$ in direction}
$\al$ is the function $f_{\al}'$ defined by
$$
f_{\al}': \ X \to T, \quad x \mapsto f(\al x) f(x)^{-1}.
$$
A function $g: X \to T$ is said to be \it{balanced} if $\sum_{x \in X} g(x) = 0$.

\begin{thm}
\label{thm-bent}
{\rm (Cf. \cite[Theorem 4.6]{fx})}
A function $f: X \to T$ is \gbent\ if and only if for any $\al \in \gi$, $f_{\al}'$ is 
balanced.
\end{thm}

\begin{thm}
\label{thm-pnf-bent}
{\rm (Cf. \cite[Theorem 5.2]{fx})}
A function $f: X \to H$ is \gpn\ if and only if for any non-principal $\xi \in \wh H$,
$\xi \circ f$ is a \gbf, i.e. (see Eqn \eqref{def-bent'})
$$
\sum_{\l \in \wh X_{\psi}} \bigabs{ \wh{(\xi \circ f)}(\l) }^2
=\frac{|X|^2}{|G|}, \quad \hbox{for all } \psi \in \wh G\setminus\{\psi_0\}.
$$
\end{thm}

Let $X$ be a $G$-set with orbits $X_1, X_2, \dotsc, X_p$, and let 
$N_i := \{ \al \in G \nid \al x = x, \hbox{ for any } x \in X_i \}$, $1 \le i \le p$.  
 Let $\wh X$ be a \gdual\ of $X$. For any subset $D$ of $X$
and any $\l \in \wh X$, let
$$
\l(D) := \{ \l(x) \nid x \in D \} \quad \hbox{and} \quad 
\l(D)^+ := \sum_{x \in D} \l(x). 
$$
Note that $\l(D)^+ = 0$ if $D = \emptyset$. 

\begin{lemma} 
\label{lem-sgd}
With the notation in the above paragraph, 
there exists a \gdual\ $\wh X$  that satisfies the following conditions.
\begin{enumerate}
\item[$($i$)$] For any $\l \in \wh X$, there is exactly
one orbit $X_j$ such that $\l(X_j) \ne \{ 0 \}$.

\item[$($ii$)$] If $\l \in \wh X$ and $\l(X_j) \ne \{ 0 \}$ for some orbit $X_j$, 
then for any $x \in X_j$, $\abs{\l(x)} = \sqrt{\abs{X}/\abs{X_j}}$.

\item[$($iii$)$] For any non-principal \ic\ $\psi$ of $G$ and any
$\l \in \wh X_{\psi}$, $\l(X)^+ = 0$.

\item[$($iv$)$] For any $X_j$ and any $\psi \in \wh G$, 
 there is at most one $\l \in \wh{X}_{\psi}$ such that $\l(X_j) \ne \{ 0 \}$.

\item[$($v$)$] For any $\psi \in \wh G$, there is $\l \in \wh X_{\psi}$
such that $\l(X_j) \ne \{ 0 \}$ for some $X_j$ 
if and only if $N_j \subseteq \ker \psi$.

\item[$($vi$)$] 
For any $X_j$, there is exactly one $\l \in \wh{X}_{\psi_0}$ 
such that $\l(X_j) \ne \{ 0 \}$ (hence $\l(x) = \sqrt{\abs{X}/\abs{X_j}}$ 
for any $x \in X_j$), where $\psi_0$ is the principal \ic\ of $G$.
\end{enumerate} 
\end{lemma}

\begin{proof}
For any $1 \le j \le p$, let $\wh{G}_j := \{ \psi \in \wh G \nid \ker \psi \supseteq N_j \}$.
Then $\wh{G}_j$ is a subgroup of $\wh G$ and is 
isomorphic to the dual group $\wh{G/N_j}$ of the quotient group $G/N_j$. Fix $x_{j0} \in
X_j$. For any $\psi \in  \wh{G}_j$, define
$$
\l_{j \psi} : X \to \c, \quad x \mapsto 
\begin{cases}
\sqrt{\abs{X}/\abs{X_j}}\,\psi(\al^{-1}), & 
             \hbox{if } x = \al x_{j0} \hbox{ for some } \al \in G; \\
0, & \hbox{if } x \notin X_j.
\end{cases}
$$
Note that if $x = \al x_{j0} = \beta x_{j0}$ for some $\al, \beta \in G$, then $\al \beta^{-1}
\in N_j$, and hence $\psi(\al^{-1}) = \psi(\beta^{-1})$ for any $\psi \in \wh{G}_j$. Thus, 
$\l_{j \psi}$ is well defined. It is clear that $\l_{j \psi}$ is $\psi$-linear. For any 
$\psi, \varphi \in \wh{G}_j$, the isomorphism of   $\wh{G}_j$ and $\wh{G/N_j}$
and the First Orthogonality Relation of \ic s imply that 
$$
\sum_{\al N_j  \in G/N_j} \psi(\al) \ol{\varphi(\al)} = 
\begin{cases}
\abs{G/N_j}, & \hbox{if } \psi = \varphi; \\
0,  & \hbox{if } \psi \ne \varphi.
\end{cases}
$$
Note that $\abs{X_j} = \abs{G}/\abs{N_j}$. So for any $\psi, \varphi \in \wh{G}_j$,
$$
\lrg{\l_{j \psi}, \l_{j \varphi}} = \sum_{x \in X_j}  \l_{j \psi}(x) \ol{\l_{j \varphi}(x)}
= \frac{\abs{X}}{\abs{X_j}} \sum_{\al N_j  \in G/N_j} \psi(\al) \ol{\varphi(\al)} 
= \begin{cases}
\abs{X}, & \hbox{if } \psi = \varphi; \\
0,  & \hbox{if } \psi \ne \varphi.
\end{cases}
$$
Let $\wh X := \bigcup_{j =1}^p \{ \l_{j \psi} \nid  \psi \in  \wh{G}_j \}$. Then it is
clear that $\wh X$ is a \gdual\ that satisfies all the conditions (i) -- (vi) of the lemma. 
(Note that the principal \ic\ $\psi_0 \in \wh{G}_j$ for all $j$.) 
\end{proof}

\begin{defn} \rm
\label{defn-sgd}
Any \gdual\ $\wh X$ that satisfies all the conditions in Lemma \ref{lem-sgd} 
is called a \it{\sgd}. If $\wh X$ is a \sgd, then the \it{support} of 
$\l \in \wh X$ is the unique orbit $X_j$ such that $\l(X_j) \ne \{ 0 \}$. 
\end{defn}

For any subset $K$ of $G$, let $K^{(-)} := \{ \al^{-1} \nid \al \in K \}$, and
$K^+ = \sum_{\al \in K} \al \in \c G$, where $\c G$ is the group algebra of $G$ 
over $\c$. Note that $K^+ =
0$ if $K = \emptyset$. For any $(x, y) \in X \times X$, let
\begin{equation}
\label{eq-gxy}
G_{x, y} := \{ \al \in G \nid \al x = y \}.
\end{equation}
Note that $G_{y, x} = G_{x, y}^{(-1)}$.

Any $\psi \in \wh G$ can be linearly extended to a function (also denoted by $\psi$)
$\psi: \c G \to \c$.
The following is a generalization of the second orthogonality relation
(see \cite[Lemma 2.5]{fx}).

\begin{lemma}
\label{lem-basic}
Let $X$ be a $G$-set,  
$\wh X$ a \sgd\ of $X$,  and $\psi \in \wh G$. Then the following hold.
\begin{enumerate}
\item[$($i$)$] For any $(x, y) \in X \times X$,
\begin{equation}
\label{eq-lxy}
\sum_{\l \in \wh X_{\psi}} \l(x) \ol{\l(y)} = 
\frac{\abs{X}}{\abs{G}} \,\psi ( G^+_{x, y}).
\end{equation}
\item[$($ii$)$] For any subsets $C, D$ of $X$,
\begin{equation}
\label{eq-lcd}
\sum_{\l \in \wh X_{\psi}} \l(C)^+ \ol{\l(D)^+} = 
\frac{\abs{X}}{\abs{G}} \,\psi \left( \sum_{(x, y) \in C \times D} G^+_{x, y} \right).
\end{equation}
\end{enumerate}
\end{lemma}
 
\begin{proof}
(i) If $x$ and $y$ are not in the same orbit of $G$, then $G^+_{x, y} = 0$, and 
for any $\l \in \wh X$, $\l(x) \ol{\l(y)} = 0$ by Lemma \ref{lem-sgd}(i).
So Eqn \eqref{eq-lxy} hols. Now assume that 
$x$ and $y$ are in the same orbit $X_j$. Then there is $\al \in G$ such that $x = 
\al y$, and hence $G_{x, y} = N_j \al$. If for all $\l \in \wh X_{\psi}$, $\l(X_j)
= \{ 0 \}$, then $\sum_{\l \in \wh X_{\psi}} \l(x) \ol{\l(y)} = 0$. Also by 
Lemma \ref{lem-sgd}(v),  $N_j \not\subseteq \ker \psi$. So $\psi |_{N_j}$
is not the principal \ic\ of $N_j$, and hence $\psi(N_j^+) = 0$. Thus, 
$\psi(G_{x, y}^+) = \psi(N_j^+) \psi(\al) = 0$, and Eqn \eqref{eq-lxy} holds.
 If there is a $\l \in \wh{X}_{\psi}$ such that $\l(X_j) \ne \{ 0 \}$, then $N_j 
\subseteq \ker \psi$, and $\al x = y$ implies that $\l(x) = \l(\al^{-1} y) =
\psi(\al)\l(y)$. 
So $\psi(G_{x, y}^+) = \psi(N_j^+) \psi(\al) = \abs{N_j} \psi(\al)$, 
and Lemma \ref{lem-sgd}(ii) yields that,
$$
\l(x) \ol{\l(y)} = \psi(\al) \abs{\l(y)}^2 =
\frac{\abs{X}}{\abs{X_j}} \psi(\al) = 
\frac{\abs{X}}{\abs{X_j}}\frac{1}{\abs{N_j}} \psi(G_{x, y}^+) =
\frac{\abs{X}}{\abs{G}} \psi(G_{x, y}^+).
$$
Note that in this case there is only one $\l \in \wh{X}_{\psi}$ such that $\l(X_j) \ne \{ 0 \}$
by Lemma \ref{lem-sgd}(iv). So Eqn \eqref{eq-lxy} holds.

(ii) follows from (i) directly.
\end{proof}

%%%%%%%%%%%%%%%%%%%%%%%%
\section{\large $G$-\dset s
\label{sect-set} }
%%%%%%%%%%%%%%%%%%%%%%%%

As in Section \ref{sect-pre}, we assume that both $G$ and $H$ are abelian. 
In this section we first study characterizations
of $G$-\dset s of $X$ via a \sgd\ $\wh X$ of $X$. Then we describe 
the relations between $G$-\dset s and \gpnf s by the \four s on $\wh X$.
As direct consequences, we will obtain some known results in \cite{dp, ph}.

The next lemma is straightforward.

\begin{lemma}
\label{lem-dset} Let $D$ be a nonempty subset of $X$. 
Then $D$ is a $G$--$(v, k, \ell)$ \dsx\ if and only if
$$
\sum_{(x, y) \in D \times D} G_{x, y}^+ = k \cdot 1_G + \ell (\gi)^+
=(k-\ell)1_G+\ell G^+, 
$$
where $G_{x, y}$ is the same as in Eqn \eqref{eq-gxy}.
\end{lemma} 

For any subsets $C, D$ of $X$, let 
\begin{equation}
\label{eq-delta-cd}
\delta(C, D) := \sum_{(x, y) \in C \times D} \abs{G_{x, y}}. 
\end{equation}
It is clear that $\delta(C, D) = \delta(D, C)$.  Note that for any $x \in X$, 
$\sum_{y \in X} \abs{G_{x, y}} = \abs{G}$. Hence $\delta(C, X) = \abs{C} \cdot 
\abs{G}$. If $D$ is a $G$--$(v, k, \ell)$ \dset, then 
$\delta(D, D) = k + \ell(\abs{G} - 1)$ by Lemma \ref{lem-dset}, and 
$\delta(D, X \setminus D) = \delta(D, X) - \delta(D, D) = (k - \ell)(\abs{G} - 1)$.

The next lemma is known but does not appear in the standard references in the 
representation theory of finite groups. We need this result in this section and
the next section.

\begin{lemma}
\label{lem-zero}
Let $a \in \c G$ and $\gamma,\mu \in \c$. Then the following hold.
\begin{enumerate}
\item[$($i$)$] $a =\mu \cdot 1_G + \gamma G^+$
if and only if for any non-principal $\psi \in \wh G$, $\psi(a) = \mu$.
\item[$($ii$)$] $a = \mu \cdot 1_G$ 
if and only if for all $\psi\in\wh G$, $\psi(a) = \mu$.
\end{enumerate}
\end{lemma}

We have a characterization of $G$-difference sets in terms of $G$-dual sets.

\begin{thm}
\label{thm-set-char}
Let $\wh X$ be a \sgd. Then a nonempty subset $D$ of $X$ 
is a $G$--$(v, k, \ell)$ \dsx\ if and only if
for any non-principal $\psi \in \wh G$,
\begin{equation}
\label{eq-l-set}
 \sum_{\l \in \wh X_{\psi}} \bigabs{\l(D)^+}^2 = \frac{\abs{X}}{\abs{G}}(k - \ell).
\end{equation}  
\end{thm}

\begin{proof}
For any $\psi \in \irr(G)$, Lemma \ref{lem-basic}(ii) implies that 
$$
 \sum_{\l \in \wh X_{\psi}} \bigabs{\l(D)^+}^2
=  \sum_{\l \in \wh X_{\psi}} \l(D)^+ \ol{\l(D)^+} =
\frac{\abs{X}}{\abs{G}} \, \psi \left( \sum_{(x, y) \in D \times D} G_{x, y}^+ \right).
$$
Thus,  if $D$ is a $G$--$(v, k, \ell)$ \dset, then for any non-principal $\psi \in \wh G$,
$\psi \left( \sum_{(x, y) \in D \times D} G_{x, y}^+ \right)$ $ = k - \ell$
by Lemma \ref{lem-dset}. Hence, Eqn \eqref{eq-l-set} holds.
 On the other hand, assume that for any non-principal $\psi \in \wh G$,
Eqn \eqref{eq-l-set} holds.
Then $\psi \left( \sum_{(x, y) \in D \times D} G_{x, y}^+ \right)$ $ = k - \ell$,
and hence  by Lemma \ref{lem-zero},
$$
\sum_{(x, y) \in D \times D} G_{x, y}^+  =  (k - \ell) \cdot 1_G + \g G^+, \quad
\hbox{ for some } \g \in \c.
$$ 
Since $1_G \in G_{x, x}$ for any $x \in X$, comparing the coefficient of $1_G$ in both 
sides of the above equality, we see that $\g = \ell$ and 
$\sum_{(x, y) \in D \times D} G_{x, y}^+  =  k \cdot 1_G + \ell (\gi)^+$. 
 So $D$ is a $G$--$(v, k, \ell)$ \dset\ by Lemma \ref{lem-dset}.
\end{proof}

Let $X_1, \dotsc, X_p$ be the orbits of the action of $G$ on $X$. 
Let $D$ be a $G$-\dsx, $C_i = D \cap X_i$ or $X_i \setminus (D \cap X_i)$,
$1 \le i \le p$, and $C = \bigcup_{i=1}^p C_i$. Let $\wh X$ be a \sgd\ of $X$.
Then for any non-principal $\psi \in \wh G$ and $\l \in \wh X_{\psi}$, 
since there is only one orbit $X_j$ such that $\l(X_j) \ne \{ 0 \}$ and 
$\l(X_j)^+ = 0$ by Lemma \ref{lem-sgd}, we see that
$\l(D)^+ = \l(D \cap X_j)^+ = - \l(X_j \setminus (D \cap X_j))^+ =
\pm \l(C_j)^+ = \pm \l(C)^+$. Hence, $C$ is also a $G$-\dset\
by Theorem \ref{thm-set-char}. 
A result similar to this observation for arbitrary \fg s was proved in 
\cite[Theorem 4.3]{dp}. 

From Theorems \ref{thm-pnf-bent} and \ref{thm-2-3} (or Corollary \ref{cor-dp-2-pnf}),
 we have the following corollary immediately.

\begin{cor}
Let $G$ be a \fag\ acting on a finite set $X$. Let $\wh X$ be a \sgd\ of $X$, and
$f: X \to \f_2$ a function. Then $f^{-1}(1)$ is a $G$--$(v, k, k - v/4)$ \dset\
if and only if 
 for any non-principal $\psi \in \wh G$ and non-principal \ic\ $\xi$ of $\f_2$,
$\sum_{\l \in \wh X_{\psi}} \bigabs{\wh{(\xi \circ f)}(\l)}^2 = 
\abs{X}^2 / \abs{G}$.
\end{cor}

The method developed in this section is useful in other places.
For example, we can easily prove the following

\begin{cor}
{\rm (Cf. \cite[Theorem 6]{ph})}
\label{cor-reg}
Let $G$ be a \fag\ acting regularly on a set $X$, and $f: X \to \f_2$ a function.
Then $f$ is $G$-\pn\ if and only if $f^{-1}(1)$ is a
$G$--$(4u^2, 2u^2 \pm u, u(u \pm 1))$ \dsx, where $\abs{X} = 4u^2$.
\end{cor}

\begin{proof}
Let $S_1 := f^{-1}(1)$, $v := \abs{X}$, and $k := \abs{S_1}$. 
Since $G$ acts regularly on $X$, we have $\abs{G} = 
\abs{X} = v$, and $\delta(S_1, S_1) = \abs{S_1}^2 = k^2$.
If $f$ is $G$-\pn, then  $S_1$ is a $G$--$(v, k, k - v/4)$ \dsx, and hence 
$\delta(S_1, S_1) = k + (k - v/4)(v - 1)$ by 
 Lemma \ref{lem-dset}. 
Thus, $k^2 = k + (k - v/4)(v - 1)$, and $k = (v \pm \sqrt{v})/2$.
But $4 | v$ implies that 
 $\pm \sqrt{v}/2 = k - v/2$ must be an integer. Let $u =  \sqrt{v}/2$.
Then $v = 4u^2$, $k = 2u^2 \pm u$, $k - v/4 = u(u \pm 1)$, and  
$S_1$ is a $G$--$(4u^2, 2u^2 \pm u, u(u \pm 1))$ \dset. 
The other direction is clear.
\end{proof}

\begin{re} \rm
Let $G$ be a \fag\ acting on a finite set $X$, and let $f: X \to \f_2$ be a $G$-\pnf.
If the action of $G$ on $X$ is not regular, then it is not necessarily $\abs{X} = 4u^2$
for some positive integer $u$. The next theorem gives such an example. Also see
Section \ref{sect-klein} for more such examples. 
\end{re}

\begin{thm}
\label{thm-order2}
Let $G$ be a group of order $2$ acting on a finite set $X$. Assume that the 
orbits of $G$ are $X_1, \dotsc, X_r, X_{r+1}, \dotsc, X_{r+s}$ such that the 
length of each $X_i$ is $2$ for $1 \le i \le r$, and the length of each $X_j$ is $1$ for
$r + 1 \le j \le r + s$. Then the following hold.
\begin{enumerate}
\item[$($i$)$] There is a $G$-\pnf\ $X \to \f_2$ if and only if $2r \ge s$ and 
$4 | (2r + s)$.
\item[$($ii$)$]
Assume that $2r \ge s$ and $4 | (2r + s)$. Then a function $f: X \to \f_2$ 
is $G$-\pn\ if and only if 
$$
\bigabs{ \{ i \nid 1 \le i \le r \hbox{ and } \abs{X_i \cap f^{-1}(1)} = 1 \} }
  = \frac{2r + s}{4}. 
$$
\end{enumerate} 
\end{thm}

\begin{proof}
Let $\wh G = \{ \psi_0, \psi_1 \}$, where $\psi_0$ is the principal \ic.  Let
$\wh X = \{ \l_i \nid 1 \le i \le r +s \} \cup \{ \eta_j \nid 1 \le j \le r \}$ be a 
\sgd\ of $X$ such that $\wh X_{\psi_1} =  \{ \eta_j \nid 1 \le j \le r \}$. 
Then for $1 \le j \le r$,
$\eta_j$ can be chosen such that the support of $\eta_j$ is $X_j$, and
$\eta_j(x_j) = \sqrt{(2r+s)/2},  \eta_j(y_j) = -\sqrt{(2r+s)/2}$, where 
$\{x_j, y_j \} = X_j$. Let
$f: X \to \f_2$ be a function, and $\xi$ the non-principal \ic\ of $\f_2$. 
Then 
$$
\wh{(\xi \circ f)}(\eta_j) = \xi(f(x_j)) \eta_j(x_j) + \xi(f(y_j)) \eta_j(y_j) =
\begin{cases}
\pm 2 \sqrt{(2r+s)/2}, & \hbox{if } \abs{X_j \cap f^{-1}(1)} = 1; \\
0, & \hbox{otherwise}.
\end{cases}
$$
Thus,
$$
\sum_{\eta \in \wh X_{\psi_1}} \lrabs{\wh{(\xi \circ f)}(\eta)}^2
=\sum_{j = 1}^r \lrabs{\wh{(\xi \circ f)}(\eta_j)}^2 = 2p(2r+s),
$$
where $p = \bigabs{ \{ j \nid 1 \le j \le r \hbox{ and } \abs{X_j \cap f^{-1}(1)} = 1 \} }$.
So by Theorem \ref{thm-pnf-bent}, $f$ is $G$-\pn\ if and only if $p = (2r + s)/4 \le r$,
and the theorem holds.
\end{proof}

When $s = 2r$ in the above theorem, we have the following 

\begin{cor}
{\rm (Cf. \cite[Theorem 4.4]{dp})}
Let $m$ be a positive integer, and let $X$ and $Y$ be two sets such that $\abs{X} = 
\abs{Y} = 2m$ and $X \cap Y = \emptyset$. Let $\pi$ be a permutation on $X \cup Y$
such that $\pi(x) = x$ for all $x \in X$, $\pi(y) \ne y$, and $\pi(\pi(y)) = y$, 
 for all $y \in Y$. Then there is a $\lrg{\pi}$-\pnf\ $f: X \cup Y \to \f_2$. Furthermore,
$f$ can be chosen to be evenly-balanced and hence $f^{-1}(1)$ is a
 $\lrg{\pi}$--$(4m, 2m, m)$ \dset\ of $X \cup Y$.
\end{cor}

%%%%%%%%%%%%%%%%%%%%%%%%%%%%%%%%%%
\section{\large Perfect nonlinear functions on actions of Klein group 
\label{sect-klein} }
%%%%%%%%%%%%%%%%%%%%%%%%%%%%%%%%%%

In this section we assume that $G := \{1_G,\alpha,\beta,\gamma\}$ 
is the Klein group, and
 study necessary and sufficient conditions under which there 
exist \gpnf s on $G$-sets. Constructions of \gpnf s are also presented. 
Note that 
$
\alpha^2=\beta^2=\gamma^2 = 1_G,~
            \alpha\beta=\gamma,~ \beta\gamma=\alpha,~
            \gamma\alpha=\beta.
$
The dual group $\wh G :=\{\psi _0, \psi_1,\psi_2,\psi_3\}$
is given by Table \ref{t-char}.
%%%%%%%%%%%%%%%%%%%%%%%%%
\begin{table}[h]
\begin{center}
$\begin{array}{|c|c c c c|}\hline
 & 1 & \alpha & \beta & \gamma \\ \hline
\psi_0 & 1 & 1 & 1 & 1\\ 
\psi_1 & 1 & 1 &-1&-1\\
\psi_2 & 1 & -1 & 1 & -1 \\ 
\psi_3 & 1 & -1 & -1 & 1 \\ \hline
\end{array}$
\caption{Character Table of the Klein Four Group \label{t-char}}\end{center}
\end{table}
%%%%%%%%%%%%%%%%%%%%%%%
\begin{nota} \rm
\label{nota-setup}
Let $G$ be the Klein four group acting on a finite set $X$ with orbits
$$
X_1, \dotsc, X_s, \ \ X_{\al, 1}, \dotsc, X_{\al, p}, \ \ X_{\beta, 1}, \dotsc, X_{\beta, q}, \
\ X_{\g, 1}, \dotsc, X_{\g, r}, \ \ X_{01}, \dotsc, X_{0t}
$$
such that
\begin{eqnarray*}
&  \abs{X_1} = \dotsb =  \abs{X_s} = 4, \ \ \abs{X_{01}} = \dotsb = \abs{X_{0t}} = 1, \\
&  \abs{X_{\al, 1}} = \dotsb = \abs{X_{\al, p}} =
\abs{X_{\beta, 1}} = \dotsb = \abs{X_{\beta, q}} =
\abs{X_{\g, 1}} = \dotsb = \abs{X_{\g, r}} = 2,  
\end{eqnarray*}
and $\al$ fixes every point in $X_{\al, 1}, \dotsc, X_{\al, p}$,  
$\beta$ fixes every point in $X_{\beta, 1}, \dotsc, X_{\beta, q}$, 
$\g$ fixes every point in $X_{\g, 1}, \dotsc, X_{\g, r}$. However, we do not assume that
$p, q, r, s, t$ are all nonzero. If $s > 0$, then for $1 \le i \le s$, assume that
$X_i := \{ x_{i1}, x_{i2}, x_{i3}, x_{i4} \}$, $\al(x_{i1}) = x_{i2}$, 
$\al(x_{i3}) = x_{i4}$, $\beta(x_{i1}) = x_{i3}$, $\beta(x_{i2}) = x_{i4}$, and
$\g(x_{i1}) = x_{i4}$, $\g(x_{i2}) = x_{i3}$.
\end{nota}

With the notation in Notation \ref{nota-setup}, we can choose a \sgd\ $\wh X$ of $X$ 
such that 
\begin{equation}
\label{eq-4-xhat}
\begin{split}
 \wh X_{\psi_1} := \{ \l_{11}, \dotsc, \l_{1s}, & \, \eta_{11}, \dotsc, \eta_{1p} \}, \ \
\wh X_{\psi_2} := \{ \l_{21}, \dotsc, \l_{2s}, \eta_{21}, \dotsc, \eta_{2q} \}, \\
& \wh X_{\psi_3} := \{ \l_{31}, \dotsc, \l_{3s}, \eta_{31}, \dotsc, \eta_{3r} \}, 
\end{split}
\end{equation}
where $\l_{ji}$ and $\eta_{jl}$ are defined as follows. For $1 \le i \le s$, 
the supports of $\l_{1i}$, $\l_{2i}$, $\l_{3i}$ are $X_i$, and 
$\l_{1i}$, $\l_{2i}$, $\l_{3i}$ are given by Table \ref{table-2}.
Furthermore, the support of $\eta_{1l}$ is $X_{\al, l} := \{ y_{l1}, y_{l2} \}$, and 
$\eta_{1l}(y_{l1}) = - \eta_{1l}(y_{l2}) = \sqrt{\abs{X}/2}$, $1 \le l \le p$; 
the support of $\eta_{2l}$ is $X_{\beta, l} := \{ z_{l1}, z_{l2} \}$, and 
$\eta_{2l}(z_{l1}) = - \eta_{2l}(z_{l2}) = \sqrt{\abs{X}/2}$, $1 \le l \le q$; and
the support of $\eta_{3l}$ is $X_{\g, l} := \{ w_{l1}, w_{l2} \}$, and 
$\eta_{3l}(w_{l1}) = - \eta_{3l}(w_{l2}) = \sqrt{\abs{X}/2}$, $1 \le l \le r$.
%%%%%%%%%%%%%%%%%%%%%%%%%
\begin{table}[h]
\begin{center}
$\begin{array}{|c|c c c c|}\hline
 & x_{i1} & x_{i2} & x_{i3} & x_{i4} \\ \hline
\l_{1i} & \sqrt{\abs{X}/4} \ & \sqrt{\abs{X}/4} & -\sqrt{\abs{X}/4} & -\sqrt{\abs{X}/4} 
       \rule{0pt}{14pt}\\
\l_{2i} & \sqrt{\abs{X}/4} & -\sqrt{\abs{X}/4} \ & \sqrt{\abs{X}/4} & -\sqrt{\abs{X}/4} \\ 
\l_{3i} & \sqrt{\abs{X}/4} & -\sqrt{\abs{X}/4} & -\sqrt{\abs{X}/4} \ & \sqrt{\abs{X}/4} \\ \hline
\end{array}$
\caption{Some Elements in $\wh X$ \label{table-2}}\end{center}
\end{table}

\begin{defn} \rm
With the notation in Notation \ref{nota-setup}, let $s > 0$, and
$f: X \to \f_2$ be a function. Then for $1 \le i \le s$ and $1 \le j \le 3$,
we say that $f$ is \it{$\psi_j$-split on} $X_i$ if 
$\abs{f^{-1}(1) \cap X_i} = 2$ and 
$\l_{ji}(f^{-1}(1) \cap X_i) = \{ \sqrt{\abs{X}/4}\, \}$ or $\{ -\sqrt{\abs{X}/4}\, \}$.
\end{defn}

The next lemma is straightforward. 

\begin{lemma}
\label{lem-4-fhat}
With the notation in Notation \ref{nota-setup}, let $\wh X$ be the \sgd\ of $X$ given by
Eqn \eqref{eq-4-xhat}, $f: X \to \f_2$ a function, and $\xi$ the non-principal \ic\ of $\f_2$.
Then the following hold.
\begin{enumerate}
\item[$($i$)$] For $1 \le i \le s$,
$$
\bigabs{\wh{(\xi \circ f)}(\l_{ji})}^2 =  
\begin{cases}
0, & \hbox{if $f$ is constant on } X_i; \\
\abs{X}, & \hbox{if } \bigabs{f^{-1}(1) \cap X_i} = 1 \hbox{ or } 3,
\end{cases}
\quad 1 \le j \le 3,
$$
and 
\begin{eqnarray*}
& \bigabs{\wh{(\xi \circ f)}(\l_{1i})}^2 = 4 \abs{X}, \
\wh{(\xi \circ f)}(\l_{2i}) = \wh{(\xi \circ f)}(\l_{3i}) = 0, 
\quad \hbox{if $f$ is $\psi_1$-split on } X_i, \\
& \bigabs{\wh{(\xi \circ f)}(\l_{2i})}^2 = 4 \abs{X}, \
\wh{(\xi \circ f)}(\l_{1i}) = \wh{(\xi \circ f)}(\l_{3i}) = 0,   
\quad \hbox{if $f$ is $\psi_2$-split on } X_i, \\
& \bigabs{\wh{(\xi \circ f)}(\l_{3i})}^2 = 4 \abs{X}, \
\wh{(\xi \circ f)}(\l_{1i}) = \wh{(\xi \circ f)}(\l_{2i}) = 0, 
\quad \hbox{if $f$ is $\psi_3$-split on } X_i.
\end{eqnarray*}

\item[$($ii$)$] 
\begin{eqnarray*}
& \bigabs{\wh{(\xi \circ f)}(\eta_{1l})}^2 =  
\begin{cases}
0, & \hbox{if } \bigabs{f^{-1}(1) \cap X_{\al, l}} = 0 \hbox{ or } 2; \\
2\abs{X}, & \hbox{if } \bigabs{f^{-1}(1) \cap X_{\al,l}} = 1,
\end{cases} \quad 1 \le l \le p, \\
& \bigabs{\wh{(\xi \circ f)}(\eta_{2l})}^2 =  
\begin{cases}
0, & \hbox{if } \bigabs{f^{-1}(1) \cap X_{\beta, l}} = 0 \hbox{ or } 2; \\
2\abs{X}, & \hbox{if } \bigabs{f^{-1}(1) \cap X_{\beta, l}} = 1,
\end{cases}  \quad 1 \le l \le q, \\
& \bigabs{\wh{(\xi \circ f)}(\eta_{3l})}^2 =  
\begin{cases}
0, & \hbox{if } \bigabs{f^{-1}(1) \cap X_{\g, l}} = 0 \hbox{ or } 2; \\
2\abs{X}, & \hbox{if } \bigabs{f^{-1}(1) \cap X_{\g, l}} = 1,
\end{cases} \quad 1 \le l \le r. 
\end{eqnarray*}
\end{enumerate}
\end{lemma}

With the notation in Notation \ref{nota-setup}, for a function $f: X \to \f_2$,
 let
\begin{equation}
\label{eq-mu-al}
\begin{split}
\mu_\al := \bigabs{\{ X_{\al, i} \nid \abs{f^{-1}(1) \cap X_{\al, i}} = 1, 1 \le i \le p \} }, 
\\
\mu_\beta := \bigabs{\{ X_{\beta, i} \nid \abs{f^{-1}(1) \cap X_{\beta, i}} = 
1, 1 \le i \le q \} }, \\
\mu_\g := \bigabs{\{ X_{\g, i} \nid \abs{f^{-1}(1) \cap X_{\g, i}} = 1, 1 \le i \le r \} }.
\end{split}
\end{equation}
Note that $\mu_\al = 0$ if $p = 0$, $\mu_\beta = 0$ if $q = 0$, and $\mu_\g = 0$ if
$r = 0$.

\begin{thm}
\label{thm-pnf}
With the notation in Notation \ref{nota-setup} and Eqn \eqref{eq-mu-al}, 
if $s = 0$, then the following hold.
\begin{enumerate}
\item[$($i$)$] There exists a $G$-\pnf\ from $X$ to $\f_2$ if and only if $8$ divides 
$\abs{X}$ and $\min \{ p, q, r \} \ge \abs{X} /8.$

\item[$($ii$)$] A function $f: X \to \f_2$ is $G$-\pn\ if and only if 
$\mu_\al = \mu_\beta = \mu_\g = \abs{X} /8. $
\end{enumerate}
\end{thm}

\begin{proof}
Let $f: X \to \f_2$ be a function, and $\xi$ the non-principal \ic\ of $\f_2$. 
Let $\wh X$ be the \sgd\ of $X$ given by Eqn \eqref{eq-4-xhat}.
Then by Lemma \ref{lem-4-fhat},
$$
\sum_{\eta_{1i} \in \wh{X}_{\psi_1}} \bigabs{\wh{(\xi \circ f)}(\eta_{1i})}^2 = 
 2 \abs{X} \mu_\al, \ \
\sum_{\eta_{2i} \in \wh{X}_{\psi_2}} \bigabs{\wh{(\xi \circ f)}(\eta_{2i})}^2 =
 2 \abs{X} \mu_\beta, \ \
\sum_{\eta_{3i} \in \wh{X}_{\psi_3}} \bigabs{\wh{(\xi \circ f)}(\eta_{3i})}^2 = 
2 \abs{X} \mu_\g.
$$
Since $s = 0$, Theorem  \ref{thm-pnf-bent} and Eqn \eqref{eq-4-xhat} imply that 
$f$ is $G$-\pn\ if and only if $\mu_\al = \mu_\beta = \mu_\g = \abs{X}/8$.
So (ii) holds.

Now we prove (i). First assume that there is a $G$-\pnf\ $f: X \to \f_2$. Then by (ii), 
$8$ divides $\abs{X}$. Furthermore, since $p \ge \mu_\al$, 
$q \ge \mu_\beta$, and $r \ge \mu_\g $, it also follows from (ii) that $\min\{p, q, r\} 
\ge \abs{X}/8$. On the other hand, if $8$ divides $\abs{X}$ and $\min\{p, q, r\} 
\ge \abs{X}/8$, then it is clear that there exists a function $f: X \to \f_2$ such that
$\mu_\al = \mu_\beta = \mu_\g = \abs{X}/8$. Hence, $f$ is $G$-\pn\ by (ii), and
(i) holds. 
\end{proof}

From the proof of Theorem \ref{thm-pnf}, we have the following corollary by 
Theorem \ref{thm-set-char}.

\begin{cor}
With the notation in Notation \ref{nota-setup} and Eqn \eqref{eq-mu-al}, let 
$f: X \to \f_2$ be a function. If $s = 0$, then $f^{-1}(1)$ is a $G$-\dsx\ if and only
if $\mu_\al = \mu_\beta = \mu_\g$.
\end{cor}

In general when $s \ne 0$, the classification of \gpnf s from $X$ to $\f_2$ is much
more complicated. The next example gives the classification for $\abs{X} = 16$.

\begin{ex} \rm
With the notation in Notation \ref{nota-setup}, let $f: X \to \f_2$ be a \gpnf.
Assume that $\abs{X} = 16$. Then one of the following hold.
\begin{enumerate}
\item[$($i$)$] $s = 4$, and either the restriction of $f$ to each $X_i$ is \gpn\
or by renumbering $X_i$ if necessary, $f$ is constant on
$X_4$ and $\psi_i$-split on $X_i$, $1 \le i \le 3$. 

\item[$($ii$)$] $s = 3$, and by renumbering $X_i$ if necessary, $f$ is 
 $\psi_i$-split on $X_i$, $1 \le i \le 3$, and $f$ is constant on other orbits of $G$.

\item[$($iii$)$] $s = 3$, $p = 2$,  
and by renumbering $X_i$ if necessary, $f$ is constant on
$X_1$ and $\psi_i$-split on $X_i$ for $i = 2, 3$, and $\mu_\al = 2$.

Similar results for cases $s = 3$, $q = 2$ and $s = 3$, $r = 2$ also hold.
 
\item[$($iv$)$] $s = 2$, and the restrictions of $f$ to $X_1$, $X_2$, and
$X \setminus (X_1 \cup X_2)$ are all \gpn.

\item[$($v$)$] $s = 2$, $q = r = 2$,  
and by renumbering $X_1, X_2$ if necessary, $f$ is $\psi_1$-split on $X_1$
and constant on $X_2$, and $\mu_\beta = \mu_\g = 2$. 

Similar results for cases $s = 2$, $p = q = 2$ and $s = 2$, $p = r = 2$ also hold.

\item[$($vi$)$] $s = 2$, $p \ge 2$,  
and by renumbering $X_1, X_2$ if necessary, $f$ is $\psi_2$-split on $X_1$
and $\psi_3$-split on $X_2$, and $\mu_\al = 2$, $\mu_\beta = \mu_\g = 0$. 

Similar results for cases $s = 2$, $q \ge 2$ and $s = 2$, $r \ge 2$ also hold.

\item[$($vii$)$] $s = 1$, $\min \{ p, q \} \ge 2$, $f$ is $\psi_3$-split on
$X_1$, $\mu_\al = \mu_\beta = 2$, and $\mu_\g = 0$.

Similar results for cases $s = 1$, $\min \{ p, r \} \ge 2$ and $s = 1$, 
$\min \{ q, r \} \ge 2$ also hold.

\item[$($viii$)$] $s = 1$, $p = q = r = 2$, $f$ is constant on $X_1$, and 
$\mu_\al = \mu_\beta = \mu_\g = 2$.

\item[$($ix$)$] $s = 0$, $\min \{ p, q, r \} \ge 2$, and $\mu_\al = \mu_\beta = \mu_\g = 2$.
\end{enumerate}
\end{ex}

Let $G$ be a \fg\ acting on a finite set $X$. A nonempty subset $Y$ of $X$ is said to
be \it{$G$-stable} if $Y$ is the union of orbits of $G$. The next lemma is 
straightforward.

\begin{lemma}
\label{lem-union}
Let $G$ be a \fg\ acting on a finite set $X$, and let $Y, Z$ be $G$-stable subsets
of $X$ such that $Y \cup Z = X$. Let $H$ be a \fg, and $f: X \to H$ a function. Then
the following hold.
\begin{enumerate}
\item[$($i$)$] If the restrictions of $f$ to both $Y$ and $Z$ are \gpn, then $f$ is 
\gpn.

\item[$($ii$)$] If $f$ and the restriction of $f$ to $Y$ are \gpn, then 
the restrictions of $f$ to $Z$ is also \gpn.
\end{enumerate} 
\end{lemma}

\begin{thm}
\label{thm-pnf-gen}
With the notation in Notation \ref{nota-setup}, there exists a \gpnf\ from $X$ to $\f_2$
if and only if one of the following hold:
\begin{enumerate}
\item[$($i$)$] There exists a \gpnf\ from $X \setminus (X_1 \cup \dotsb \cup X_s)$
to $\f_2$;  or
\item[$($ii$)$] $4| v_0$, where $v_0 = 2p + 2q + 2r + t$, and
there are nonnegative integers $k_0, p_1, q_1, r_1$ such that $p_1 \le p$,
$q_1 \le q$, $r_1 \le r$,  and
\begin{eqnarray*}
&  & p_1 \equiv q_1 \equiv r_1 \mod 2, \quad
s \ge \frac{3 v_0}{4} + 4 k_0 - 2(p_1 + q_1 + r_1), \\
 & & \frac{v_0}{4} + k_0 \ge \max \Big\{ p_1 + \frac{q_1 + r_1}{2}, 
q_1 + \frac{p_1 + r_1}{2},  r_1 + \frac{p_1 + q_1}{2} \Big\}.
\end{eqnarray*}
\end{enumerate}
\end{thm}

\begin{proof}
First assume that there exists a \gpnf\ from $X$ to $\f_2$. Then we prove that (i) or (ii) holds.
If (i) holds, there is nothing to prove. Assume that (i) does not hold. Then there is a minimal
number of subsets among $X_1, \dotsc, X_s$, say $X_1, \dotsc, X_{s_1}$, such that 
there exists a \gpnf\ $f: X_0 \cup X_1 \cup \dotsb \cup X_{s_1} \to \f_2$, 
where $X_0 = X \setminus (X_1 \cup \dotsb \cup X_s)$, and $1 \le s_1 \le s$.
Then by Lemma \ref{lem-union}, the choice of the minimal number of the
subsets $X_i$ implies that the restriction of $f$ to any $X_i$ is not \gpn, $1 \le i \le s_1$.
Thus, for any $1 \le i \le s_1$, $f$ is either constant or $\psi_j$-split on $X_i$ for some 
 $1 \le j \le 3$.

Let $Y = X_0 \cup X_1 \cup \dotsb \cup X_{s_1}$, and $\wh Y$ a \sgd\ of $Y$ given by
Eqn \eqref{eq-4-xhat}. Among the subsets $X_1, \dotsc, X_{s_1}$, assume that there are $k_0$
subsets on which $f$ is constant, and $k_i$ subsets on which $f$ is  $\psi_i$-split, 
$1 \le i \le 3$. Also let $p_1 := \mu_\al$, $q_1 := \mu_\beta$, and $r_1 := \mu_\g$,
where $\mu_\al, \mu_\beta, \mu_\g$ are defined by Eqn \eqref{eq-mu-al}. 
Then by Lemma \ref{lem-4-fhat} and Theorem \ref{thm-pnf-bent},
$$
\frac{\abs{Y}^2}{4} = \sum_{\l \in \wh Y_{\psi_1}} \bigabs{ \wh{(\xi \circ f)}(\l)}^2
= (4 k_1 + 2p_1) \abs{Y}, 
$$
where $\xi$ is the non-principal \ic\ of $\f_2$. Thus, $4 k_1 + 2p_1 = \abs{Y}/4$. 
Similarly, $4 k_2 + 2q_1 = 4 k_3 + 2r_1 = \abs{Y}/4$. Thus, $4(k_1 + k_2 + k_3)
+ 2(p_1 + q_1 + r_1) = 3\abs{Y}/4$. Since $\abs{Y} = 4 s_1 + v_0$
and $s_1 = k_0 + k_1 + k_2 +k_3$, we get that 
\begin{equation}
\label{eq-s1}
s_1 = \frac{3 v_0}{4} + 4 k_0 - 2(p_1 + q_1 + r_1).
\end{equation}
Also it follows from $4 k_1 + 2p_1 = 4 k_2 + 2q_1 = 4 k_3 + 2r_1 = \abs{Y}/4$ that
$p_1 \equiv q_1  \equiv r_1 \mod 2$. Furthermore, $\abs{Y} = 4 s_1 + v_0 = 4 v_0 +
16 k_0 - 8 (p_1 + q_1 + r_1)$ by Eqn \eqref{eq-s1}. So $4 k_1 + 2p_1 = \abs{Y}/4$
yields that 
$$
k_1 = \frac{v_0}{4} + k_0 - p_1 - \frac{q_1 + r_1}{2} \ge 0.
$$
Similarly, $k_2 = v_0/4 + k_0 - q_1 - (p_1 + r_1)/2 \ge 0$, and
$k_3 = v_0/4 + k_0 - r_1 - (p_1 + q_1)/2 \ge 0$. So (ii) holds.

On the other hand, if (i) holds, then clearly there is a \gpnf\ from $X$ to $\f_2$
by Lemma \ref{lem-union}. Now assume
that (ii) holds. Let $k_1 =  v_0/4 + k_0 - p_1 - (q_1 + r_1)/2$, 
$k_2 =  v_0/4 + k_0 - q_1 - (p_1 + r_1)/2$, and 
$k_3 =  v_0/4 + k_0 - r_1 - (p_1 + q_1)/2$. Then $k_1$, $k_2$, and $k_3$ are
nonnegative integers, and
$$
k_0 + k_1 + k_2 + k_3 = \frac{3v_0}{4} + 4 k_0 - 2(p_1 + q_1 + r_1) \le s.
$$
Let $f: X \to \f_2$ be a function that satisfies the following conditions: $\mu_\al = p_1$,
$\mu_\beta = q_1$, $\mu_\g = r_1$;  among the subsets $X_1, \dotsc, X_{s_1}$, 
where $s_1 := k_0 + k_1 + k_2 + k_3$,  $f$ is constant on $k_0$ subsets and 
$\psi_i$-split on $k_i$ subsets, $1 \le i \le 3$, and $\bigabs{f^{-1}(1) \cap X_j}
= 1$ or $3$ for $s_1 < j \le s$.  Let $Y = X_0 \cup X_1 \cup \dotsb \cup X_{s_1}$,
where $X_0 = X \setminus (X_1 \cup \dotsb \cup X_s)$. Then by Lemma \ref{lem-4-fhat} 
and Theorem \ref{thm-pnf-bent}, it is easy to show that the restriction of $f$ to $Y$
is \gpn. Since the restriction of $f$ to each $X_j$ is \gpn, $s_1 < j \le s$,
$f$ is a \gpnf\ from $X$ to $\f_2$ by Lemma \ref{lem-union}.
\end{proof}

%%%%%%%%%%%%%%%%%%%%%%%%%%%%%%%%%%
\section{\large Bent functions on actions of Klein group
\label{sect-klein-b} }
%%%%%%%%%%%%%%%%%%%%%%%%%%%%%%%%%%

In this section we study \gbf s for the Klein group $G$. Our main result
is Theorem \ref{thm-klein4} below. We will prove this theorem by constructing the
desired \gbf s. Note that if there is a \gpnf\ from $X$ to a \fag, then 
there exists a \gbf\ on $X$ by Theorem \ref{thm-pnf-bent}. However, in this section 
we will construct the \gbf s directly. These functions cannot be obtained from
the \gpnf s constructed in the previous section. 

\begin{nota} \rm
\label{nota-orbit}
Let $G$ be the Klein four group acting on a finite set $X$ with orbits
$$
 X_{\al, 1}, \dotsc, X_{\al, p}, \ \ X_{\beta, 1}, \dotsc, X_{\beta, q}, \
\ X_{\g, 1}, \dotsc, X_{\g, r}
$$
such that the length of each orbit is $2$, 
 $\al$ fixes every point in $X_{\al, 1}, \dotsc, X_{\al, p}$,  
$\beta$ fixes every point in $X_{\beta, 1}, \dotsc, X_{\beta, q}$, and
$\g$ fixes every point in $X_{\g, 1}, \dotsc, X_{\g, r}$. Also let $X_{\al, i} = \{x_i, 
y_i\}$, $1 \le i \le p$,  $X_{\beta, i} = \{x_{p+i}, y_{p+i}\}$, $1 \le i \le q$, and
$X_{\g, i} = \{x_{p+q+i}, y_{p+q+i}\}$, $1 \le i \le r$.
However, we do not assume that $p, q, r$ are all nonzero. 
\end{nota}

\begin{thm}
\label{thm-klein4}
With the notation in Notation \ref{nota-orbit},
 there exist (infinitely many) bent functions on $X$ if and only if 
\begin{equation}
\label{eq-g-lin}
p+q+r \le 4 \min \{ p, q, r \}.
\end{equation}
\end{thm}

In order to prove Theorem \ref{thm-klein4}, we need a few lemmas first. In
the following we always assume that $G$ is the Klein four group whose dual group
$\wh G$ is given by Table \ref{t-char}.
With the notation in Notation \ref{nota-orbit}, we can choose a \sgd\ $\wh X$ of $X$ 
such that 
\begin{equation}
\label{eq-4-xhat-2}
\begin{split}
& \wh X_{\psi_0} := \{ \l_{11}, \dotsc, \l_{1p},  \l_{21}, \dotsc, \l_{2q},
 \l_{31}, \dotsc, \l_{3r} \}, \\
 \wh X_{\psi_1} :=  \{ & \, \eta_{11}, \dotsc, \eta_{1p} \}, 
\wh X_{\psi_2} := \{  \eta_{21}, \dotsc, \eta_{2q} \}, 
 \wh X_{\psi_3} := \{ \eta_{31}, \dotsc, \eta_{3r} \}, 
\end{split}
\end{equation}
where $\l_{ij}$ and $\eta_{ij}$ are defined as follows. 
The supports of $\l_{1i}$ and $\eta_{1i}$ are $X_{\al, i}$, 
$\l_{1i}(x_i) = \l_{1i}(y_i) = \sqrt{\abs{X}/2}$, and 
$\eta_{1i}(x_i) = - \eta_{1i}(y_i) = \sqrt{\abs{X}/2}$, $1 \le i \le p$; 
the supports of  $\l_{2i}$ and $\eta_{2i}$ are $X_{\beta, i}$, 
$\l_{2i}(x_{p+i}) = \l_{2i}(y_{p+i}) = \sqrt{\abs{X}/2}$, and 
$\eta_{2i}(x_{p+i}) = - \eta_{2i}(y_{p+i}) = \sqrt{\abs{X}/2}$, $1 \le i \le q$; and
the supports of  $\l_{3i}$ and $\eta_{3i}$ are $X_{\g, i}$, 
$\l_{3i}(x_{p+q+i}) = \l_{3i}(y_{p+q+i}) = \sqrt{\abs{X}/2}$, and
$\eta_{3i}(x_{p+q+i}) = - \eta_{3i}(y_{p+q+i}) = \sqrt{\abs{X}/2}$, $1 \le i \le r$.

The next lemma generalizes \cite[Example 6.3]{fx}.

\begin{lemma}
\label{lem-bent-6}
With the notation in Notation \ref{nota-orbit}, if $\abs{X} = 6$ and $p = q = r = 1$, 
 then there are infinitely many \gbf s on $X$.
\end{lemma}

\begin{proof}
Let $f: X \to T$ be a function,
and $c_i := f(x_i)$,  $d_i := f(y_i)$, $1 \le i \le 3$. Then
by Eqn \eqref{eq-4-xhat-2}, 
\begin{eqnarray}
\label{eq-f6-psi0}
\sum_{\l \in \wh X_{\psi_0}} \bigabs{\wh f(\l)}^2
=\sum_{i=1}^3\lrabs{\sum_{x\in X}f(x)\l_{i1}(x)}^2  
  =  3 \sum_{i=1}^3 \abs{c_i+d_i}^2 =
18+ 3 \sum_{i=1}^3 (\ol{c_i} d_i+ c_i \ol{d_i}),
\end{eqnarray}
and 
\begin{eqnarray}
\label{eq-f6-psi1}
\sum_{\eta \in \wh X_{\psi_1}} \bigabs{\wh f(\eta)}^2
=\lrabs{\sum_{x\in X}f(x)\eta_{11}(x)}^2 
 =  3 \abs{c_1 - d_1}^2 
= 6 - 3 (\ol{c_1} d_1 + c_1 \ol{d_1}).
\end{eqnarray}
Similarly,
\begin{equation}
\label{eq-f6-psi3}
\sum_{\eta \in \wh X_{\psi_2}} \bigabs{\wh f(\eta)}^2 = 
6 - 3 (\ol{c_2} d_2 + c_2 \ol{d_2}), \quad
\sum_{\eta \in \wh X_{\psi_3}} \bigabs{\wh f(\eta)}^2 =
 6 - 3 (\ol{c_3} d_3 + c_3 \ol{d_3}).
\end{equation}

Now assume that $f$ is \gbent. Then Eqn \eqref{eq-f6-psi1} and Eqn \eqref{eq-f6-psi3} imply that
$$
\ol{c_1} d_1 + c_1 \ol{d_1} = \ol{c_2} d_2 + c_2 \ol{d_2} =
\ol{c_3} d_3 + c_3 \ol{d_3}.
$$
Thus, Eqn \eqref{eq-f6-psi0} and Eqn \eqref{eq-f6-psi1} imply that 
$$
18 + 9(\ol{c_1} d_1 + c_1 \ol{d_1}) = 6 - 3 (\ol{c_1} d_1 + c_1 \ol{d_1}).
$$
So $\ol{c_1} d_1 + c_1 \ol{d_1} = -1$, and hence $\ol{c_1} d_1 = -1/2 + b \sqrt{-1}$ for some 
real number $b$. Therefore,  $\abs{\ol{c_1} d_1} = 1$ forces that $\ol{c_1} d_1 = 
-1/2 \pm \sqrt{-3}/2$, and 
$d_1 = ( -1/2 \pm \sqrt{-3}/2 )c_1$. Similarly, 
$d_2 = ( -1/2 \pm \sqrt{-3}/2 )c_2$, and $d_3 = ( -1/2 \pm \sqrt{-3}/2 )c_3$.
On the other hand, assume that
$d_1 = ( -1/2 \pm \sqrt{-3}/2 )c_1$, 
$d_2 = ( -1/2 \pm \sqrt{-3}/2 )c_2$, and $d_3 = ( -1/2 \pm \sqrt{-3}/2 )c_3$.
 Then it is easy to verify that
$$
\sum_{\l \in \wh X_{\psi_0}} \bigabs{\wh f(\l)}^2 = 9, \
\sum_{\eta \in \wh X_{\psi_j}} \bigabs{\wh f(\eta)}^2 = 9,  \quad 1 \le j \le 3.
$$
So $f$ is \gbent, and the lemma holds.
\end{proof}

\begin{lemma}
\label{lem-bent-8}
With the notation in Notation \ref{nota-orbit}, if $\abs{X} = 8$ and $p, q, r$ are all 
nonzero, then there are infinitely many \gbf s on $X$.
\end{lemma}

\begin{proof}
Without loss of generality, we may assume that $p = 2$, $q = r = 1$.
Let $f: X \to T$ be a function, and $c_i := f(x_i)$,  $d_i := f(y_i)$, $1 \le i \le 4$. Then
by Eqn \eqref{eq-4-xhat-2}, 
\begin{eqnarray}
\label{eq-f-psi0}
\sum_{\l_{ij} \in \wh X_{\psi_0}} \bigabs{\wh f(\l_{ij})}^2
=\sum_{i, j}\lrabs{\sum_{x\in X}f(x)\l_{ij}(x)}^2  
  =  4 \sum_{i=1}^4 \abs{c_i+d_i}^2 =
32+ 4 \sum_{i=1}^4 (\ol{c_i} d_i+ c_i \ol{d_i}),
\end{eqnarray}
and 
\begin{eqnarray}
\label{eq-f-psi1}
\sum_{\eta_{1i} \in \wh X_{\psi_1}} \bigabs{\wh f(\eta_{1i})}^2
=\sum_{i=1}^2 \lrabs{\sum_{x\in X}f(x)\eta_{1i}(x)}^2 
 =  4 \sum_{i=1}^2 \abs{c_i - d_i}^2 
= 16 - 4 \sum_{i=1}^2 (\ol{c_i} d_i + c_i \ol{d_i}).
\end{eqnarray}
Similarly,
\begin{equation}
\label{eq-f-psi3}
\sum_{\eta_{2i} \in \wh X_{\psi_2}} \bigabs{\wh f(\eta_{2i})}^2 = 
8 - 4 (\ol{c_3} d_3 + c_3 \ol{d_3}), \quad
\sum_{\eta_{3i} \in \wh X_{\psi_3}} \bigabs{\wh f(\eta_{3i})}^2 =
 8 - 4 (\ol{c_4} d_4 + c_4 \ol{d_4}).
\end{equation}

Now assume that $f$ is \gbent. Then Eqn \eqref{eq-f-psi3} and Eqn \eqref{eq-f-psi1} imply that
$$
\ol{c_3} d_3 + c_3 \ol{d_3} = \ol{c_4} d_4 + c_4 \ol{d_4} \quad \hbox{and} \quad
\sum_{i=1}^2(\ol{c_i} d_i+ c_i \ol{d_i}) = 2 + \ol{c_3} d_3 + c_3 \ol{d_3}.
$$
Thus, from Eqn \eqref{eq-f-psi0} we see that 
$$
40 + 12(\ol{c_3} d_3 + c_3 \ol{d_3}) = 8 - 4 (\ol{c_3} d_d + c_3 \ol{d_3}).
$$
So $\ol{c_3} d_3 + c_3 \ol{d_3} = -2$, and hence $\ol{c_3} d_3 = -1 + b \sqrt{-1}$ for some 
real number $b$. Therefore,  $\abs{\ol{c_3} d_3} = 1$ forces that $\ol{c_3} d_3 = -1$, and 
$c_3 = - d_3$. Similarly, $c_4 = -d_4$. That is, we have proved that
\begin{equation}
\label{eq-c3-c4}
\hbox{if $f$ is \gbent, then } f(x_3) = - f(y_3), \hbox{ and } f(x_4) = - f(y_4).
\end{equation}

On the other hand, let $f$ be a \gbf\ such that $f(x_3) = - f(y_3)$ and $f(x_4) = - f(y_4)$.
Furthermore, if $f(x_1) =  f(x_2)$ and $f(y_1) = - f(y_2)$, then from \eqref{eq-f-psi0} --
\eqref{eq-f-psi3}, it is easy to verify that
$$
\sum_{\l_{ij} \in \wh X_{\psi_0}} \bigabs{\wh f(\l_{ij})}^2 = 16, \
\sum_{\eta_{ji} \in \wh X_{\psi_j}} \bigabs{\wh f(\eta_{ji})}^2 = 16,  \quad 1 \le j \le 3.
$$
Hence, $f$ is a \gbf. Similarly, $f$ is also a \gbf\ if $f(x_1) = f(y_1)$, 
$f(x_2) = - f(y_2)$ or $f(x_1) = - f(y_1)$, $f(x_2) = f(y_2)$.
Therefore, there are infinitely many \gbf s on $X$.
\end{proof}

The next lemma follows directly from Theorem \ref{thm-bent}.

\begin{lemma}
\label{lem-extend}
Let $G$ be a \fag\ acting on a finite set $X$. Assume that $X$ is the disjoint union
of its $G$-stable subsets $X_i$, $1 \le i \le p$. If there exists a \gbf\ on $X_i$
for any $1 \le i \le p$, then there exists a \gbf\ on $X$.
\end{lemma}

Now we are ready to prove Theorem \ref{thm-klein4}.

\begin{proof}[Proof of Theorem \ref{thm-klein4}.]
With the notation in Notation \ref{nota-orbit}, without loss of generality, we may assume that 
$r \le p, q$. Then Eqn \eqref{eq-g-lin} is equivalent to $p + q \le 3r$.  

Assume that there is a \gbf\ $f: X \to T$. Let $c_i = f(x_i)$, and
$d_i = f(y_i)$, $1 \le i \le p + q + r$. Then by Eqn \eqref{eq-4-xhat-2},
\begin{eqnarray}
\label{eq-psi1}
\sum_{\l_{ij} \in \wh{X}_{\psi_0}} \bigabs{\wh f(\l_{ij})}^2 & = &
\sum_{i, j} \lrabs{\sum_{x\in X}f(x)\l_{ij}(x)}^2  
  =  (p + q + r) \sum_{i=1}^{p + q +r} \abs{c_i + d_i}^2  \nonumber\\
& = & 2(p + q + r)^2 + (p + q + r) \sum_{i=1}^{p + q +r}(\ol{c_i} d_i + c_i \ol{d_i}),
\end{eqnarray}
and 
\begin{eqnarray}
\label{eq-psi2}
\sum_{\eta_{1i} \in \wh{X}_{\psi_1}} \bigabs{\wh f(\l_{1i})}^2 & = &
\sum_{i=1}^p \lrabs{\sum_{x\in X}f(x)\eta_{1i}(x)}^2 
 =  (p + q + r) \sum_{i=1}^p \abs{c_i - d_i}^2 \nonumber \\
& = & 2p(p + q + r) - (p + q + r)  \sum_{i=1}^p (\ol{c_i} d_i + c_i \ol{d_i}).
\end{eqnarray}
Similarly,
\begin{equation}
\label{eq-psi3}
\sum_{\eta_{2i} \in \wh{X}_{\psi_2}} \bigabs{\wh f(\l_{2i})}^2  = 
2q(p + q + r) - (p + q + r)  \sum_{i= p + 1}^{p + q} (\ol{c_i} d_i + c_i \ol{d_i}), 
\end{equation}
and
\begin{equation}
\label{eq-psi4}
\sum_{\eta_{3i} \in \wh{X}_{\psi_3}} \bigabs{\wh f(\l_{3i})}^2  =  
2r(p + q + r) - (p + q + r)  \sum_{i=p+q+1}^{p+q+r} (\ol{c_i} d_i + c_i \ol{d_i}).
\end{equation}
Since $f$ is \gbent, it follows from \eqref{def-bent'} and
\eqref{eq-psi2} -- \eqref{eq-psi4} that
$$
\sum_{i=1}^p (\ol{c_i} d_i + c_i \ol{d_i}) = 2(p-r) + 
\sum_{i=p+q+1}^{p+q+r} (\ol{c_i} d_i + c_i \ol{d_i}), \quad
\sum_{i=p+1}^{p+q} (\ol{c_i} d_i + c_i \ol{d_i}) = 2(q-r) + 
\sum_{i=p+q+1}^{p+q+r} (\ol{c_i} d_i + c_i \ol{d_i}).
$$
Thus, Eqn \eqref{eq-psi1} and Eqn \eqref{eq-psi4} imply that 
$$
2(p + q + r) + 2(p-r) + 2(q-r) + 3 \sum_{i=p+q+1}^{p+q+r} (\ol{c_i} d_i + c_i \ol{d_i})
= 2r - \sum_{i=p+q+1}^{p+q+r} (\ol{c_i} d_i + c_i \ol{d_i}).
$$
Hence,
$$
p + q - r = - \sum_{i=p+q+1}^{p+q+r} (\ol{c_i} d_i + c_i \ol{d_i}) =
\lrabs{\sum_{i=p+q+1}^{p+q+r} (\ol{c_i} d_i + c_i \ol{d_i})} \le 2r.
$$
So Eqn \eqref{eq-g-lin} holds.

On the other hand, assume that Eqn \eqref{eq-g-lin} holds. Then by our assumption at the
beginning of the proof, $r \le p, q$ and $p + q \le 3r$. Let $A = \{ X_{\al, i} \nid 1 \le i
\le p \}$, $B = \{ X_{\beta, i} \nid 1 \le i \le q \}$, and $C = \{ X_{\g, i} \nid 1 \le i
\le r \}$. Since $r \le p, q$ and $p + q \le 3r$,
the orbits of $G$ can be divided into $r$ families as follows. 
Each family consists of three or four orbits, with one orbit from each of $A, B$, 
and $C$, and possibly one more orbit from either $A$ or $B$ if the family consists 
of four orbits. Since there exist infinitely many \gbf s on these families by
Lemmas \ref{lem-bent-6} and \ref{lem-bent-8}, 
 there exists infinitely many \gbf s on $X$ by Lemma \ref{lem-extend}.
\end{proof}

\baselineskip=14pt
\section*{\large Acknowledgments} \small

Part of this work was done when the first author
was visiting the second author at Eastern Kentucky University
in 2014. The first author is grateful for the hospitality.
The first author is supported by NSFC through the grant number 11271005.

\begin{flushleft}

\end{flushleft}

\end{document}